\documentclass[11pt]{article}
\usepackage{amsmath}
\usepackage{graphicx}
\usepackage{enumerate}
 
\usepackage{url,natbib} 

\usepackage{authblk}
\usepackage{amsmath,amsthm,mathrsfs,amsfonts}	
\usepackage{url}
\usepackage{amsfonts}
\usepackage{amssymb}
\usepackage{bbm}
\usepackage{graphicx}
\usepackage{subfig}
\usepackage[normalem]{ulem} 
\usepackage{float}
\usepackage{xcolor}
\usepackage{graphicx}
\usepackage{multirow}
\usepackage{tikz}
\usetikzlibrary{shapes,backgrounds}
\usetikzlibrary {mindmap}

\newcommand{\tstar}[5]{
\pgfmathsetmacro{\starangle}{360/#3}
\draw[#5] (#4:#1)
\foreach \x in {1,...,#3}
{ -- (#4+\x*\starangle-\starangle/2:#2) -- (#4+\x*\starangle:#1)
}
-- cycle;
}
\newcommand{\ngram}[4]{
\pgfmathsetmacro{\starangle}{360/#2}
\pgfmathsetmacro{\innerradius}{#1*sin(90-\starangle)/sin(90+\starangle/2)}
\tstar{\innerradius}{#1}{#2}{#3}{#4}
}

\newcommand{\ccs}{\mbox{$\stackrel{a.s}{\longrightarrow}$}}

\newcommand{\E}{\mathbb{E}}

\newcommand{\eps}{\varepsilon}
\newtheorem{theorem}{Theorem}[section]

\newtheorem{remark}{Remark}
\newtheorem{proposition}{Proposition}

\theoremstyle{remark}
\newtheorem{definition}[theorem]{Definition}

\begin{document}

  \title{\bf Statistical analysis of measures of non-convexity}
  \author{Alejandro Cholaquidis\thanks{
    The authors gratefully acknowledge support from grant FCE-1-2019-1-156054, ANII (Uruguay), grant PID2020-116587GB-I00 funded by MCIN/AEI/10.13039/501100011033 and ED431C 2021/24 funded by Conseller\'{i}a de Cultura, Educaci\'{o}n e Universidade.}\hspace{.2cm}\\
    Facultad de Ciencias, Universidad de la Rep\'{u}blica (Uruguay)\\
    and \\
    Ricardo Fraiman \\
    Facultad de Ciencias, Universidad de la Rep\'{u}blica (Uruguay)\\
    and \\
    Leonardo Moreno \\
    Departamento de M\'{e}todos Cuantitativos, Facultad de Ciencias Econ\'{o}micas y de Administraci\'{o}n, Universidad de la Rep\'{u}blica (Uruguay)\\
    and \\
    Beatriz Pateiro-L\'{o}pez \\
    CITMAga (Galician Center for Mathematical Research and Technology)\\ and Departamento de Estat\'{i}stica, An\'{a}lise Matem\'{a}tica  e Optimizaci\'{o}n,\\ Universidade de Santiago de Compostela (Spain)}
  \maketitle
 
\begin{abstract}
Several measures of non-convexity (departures from convexity) have been introduced in the literature, both for sets and functions. Some of them are of geometric nature, while others are more of topological nature.  We address the statistical analysis of some of these measures of non-convexity of a set $S$, by dealing with their estimation based on a sample of points in $S$. We introduce also a new measure of non-convexity. We discuss briefly about these different notions of non-convexity, prove consistency and find the asymptotic distribution for the proposed estimators. We also consider the practical implementation of these estimators and illustrate their  applicability to a real data example.
\end{abstract}

\noindent%
{\it Keywords:}  convex hull, convexity measure, shape analysis, skin lesions
\vfill

\section{Introduction} 
Convex sets and convex functions play an important role in many areas of mathematics and statistics. Convexity intrinsically combines geometry, linear algebra, analysis and topology in several ways \citep{gruber1993handbook, roc1970}. 
The notion of convexity is central, for instance, in mathematical programming, in the sub-field of convex programming \citep{boy2004}. Convex optimization problems arise in a variety of applications, including control theory, finance, or engineering, and are inherent to many statistical and machine learning methods that require the minimization of a given loss function. 
In addition to its role in the underlying optimization problems, convexity is also relevant in other areas of statistics, such as set estimation, where the goal is to approximate a set from a random sample of points taken into it \citep{cue2009}. The set of interest may be the support of a distribution or different functionals associated with it, such as its Lebesgue measure or  surface area. In this context, shape restrictions are usually imposed.

Convexity is probably one of the most important and has given rise to a considerable amount of literature on the subject. We refer to \citep{bru2018} for a survey on the estimation of convex sets.
However, in some practical applications (like, for instance, in many set estimation problems, or in optimization problems where the domain is not convex) convexity can be a quite restrictive hypothesis. In this regard, to gain an insight into how much a set or a function departures from convexity can be useful to tune the parameters, or adapt (or even change) the algorithms or the estimation methods. 

\

\noindent{\emph{Measures of non-convexity of sets.}}	Several measures of non-convexity were introduced in the literature, both for sets and functions, see \citep{fradelizi2018} for a review of some of the most commonly used and their properties. Many of these measures follow naturally from the standard definition of convexity and its usual characterizations. Recall that a  set $S \subset \mathbb R^d$  is convex if and only if it contains every line segment whose endpoints both belong to $S$. If the set is not convex, this will not always be true. Thus, we can think of the probability that, for two randomly chosen points in $S$, the line segment connecting them is contained in $S$, as a measure of non-convexity.
Extending the notion of line segment to convex combinations of points in $S$, we get an alternative and equivalent characterization of a convex set. The convex hull of $S$, denoted by $\text{co}(S)$, is defined as the set of all convex combinations of points in $S$ and turns out to be the smallest convex set containing $S$. Thus, a set is convex if and only if it is equal to its convex hull. This characterization leads to new ways of defining measures of non-convexity through, for instance, the Hausdorff distance or the difference (in measure, surface area, etc.) between the set and its convex hull. Caratheodory's Theorem  states that, given  $S \subset \mathbb R^d$,  every point in $\text{co}(S)$  can be written as a convex combination of no more than $d+1$ points in $S$ \citep{carat}. In particular, every point in $\text{co}(S)$ lies in an $d$-dimensional simplex with vertices in $S$, that is, in the convex hull of $d+1$ affinely independent points in $S$. We can use this result as starting point for the definition of a new measure of non-convexity, based on the probability that the convex hull of $d + 1$ randomly chosen points in $S$ is contained in $S$.

Some of these ideas have already been studied in the fields of computer vision and pattern recognition, where shape descriptors are of particular importance, being convexity one of the most widely used. Some previous works introduce different measures of non-convexity for polygons, such as the polygonal entropy \citep{ste1989}, the simple and total index of non-convexity  \citep{box1993}, or the area-based and perimeter-based measures of non-convexity that compare, respectively, the area and perimeter of the set with those of its convex hull \citep{zun2004}. For the case of planar sets, \citep{rahtu2006} define a measure of non-convexity as the probability that a point which lies on the line segment between two random points from the set is also contained in the set. This idea can be generalized to convex combinations of three or more points. Also for planar shapes, but with techniques that can be applied to higher dimensions, \citep{Rosin_2007} propose a new measure of non-convexity that incorporates both area-based and boundary-based information. It is defined as the probability that for two random points, chosen uniformly along the boundary of the set, all points from the line segment connecting them belong to the set. 
The notion of convexity is also related to that of visibility, that has been used extensively in the area of computational geometry \citep{rou1997}. Using the terminology of the field, a set is convex if every two points in it are visible to each other (or mutually visible), meaning that the line segment connecting them is contained in the set. Thus, \citep{ste1989} defines a new measure of non-convexity for a polygon as the normalized average visible area.  This measure, also known as {\emph{Beer index of convexity}}, was first introduced for measurable sets in $\mathbb{R}^d$ by \citep{beer1973}. The author was interested in the generalizations of convex and starshaped sets in terms of the so-called visibility function.  \citep{bal2017} also consider the Beer index of convexity and higher-order generalizations.

These indices have relevant applications in a variety of real data problems, for example in medical imaging  \citep{lee2005, lee2003}. In particular, computational systems for the automatic diagnosis of skin lesions from dermatoscopic images use several combined image features as inputs  \citep{oliveira2019}. Among these features, measures of border irregularity (including some of the aforementioned measures of non-convexity   \citep{rahtu2006,rosin2009}) are generally considered. In practice, these measures are approximated from digital images using image processing tools or Monte Carlo methods (as in the case of measures based on probabilistic ideas  \citep{rahtu2006, Rosin_2007}).  Despite the rich literature on the subject, as far as we know, no prior studies have examined the problem of the estimation of these measures from a statistical point of view. The aim of this work is to propose estimators for some of these measures of non-convexity of a set, based on a random sample of points taken into it, and investigate their theoretical properties.

Another way of measuring how far a set $S\subset\mathbb{R}^d$  is from being convex, that has been considered mainly in the literature on discrete mathematics, through the so-called convexity number, $\gamma(S)$, defined as the least number of convex subsets of $S$ that covers $S$, see \citep{koj1990}. The convexity number has been studied in connection with other two measures of non-convexity defined from the invisibility graph of $S$.
The invisibility graph of $S$, denoted by $I(S)$, is a (possible infinite) graph that has a node for every point of $S$ and an edge for every pair of points that are not visible to each other. 
Thus, we can define the clique number of a set, $\omega(S)$, and the chromatic number of a set, $\chi(S)$, as the clique number and chromatic number of its invisibility graph, respectively. 
Recall that a clique of a graph is a complete subgraph, that is, a subset of the vertices such that every two vertices are adjacent. The clique number of a graph is the number of vertices in a maximum clique (a clique having maximum number of vertices). The chromatic number of a graph is the smallest number of colours needed to colour the vertices, so that no two adjacent vertices have the same colour \citep{skiena}. 
We have that $\omega(S)\leq \chi(S)\leq \gamma(S)$. Bounds of $\gamma(S)$ from above by functions of $\omega(S)$ and $\chi(S)$ are discussed in detail in \citep{cibulka} and references therein. It is easy to derive the consistency of these two discrete measures based on a random sample of points in $S$ when they are assumed to be finite.


\

\noindent{\emph{Measures of non-convexity of functions.}} Since a function $f: \mathbb{R}^d \to \mathbb R$ is convex if and only if its epigraph is a convex subset of $\mathbb{R}^{d+1}$, the measures of non-convexity for sets extend naturally to measures of non-convexity for functions. In any case, we can also define measures of non-convexity for functions directly from the notion of convex function. 

\

\noindent{\emph{Contributions of this work.}} As mentioned above, and despite the considerable literature on measures of non-convexity, we are not aware of previous works dealing with the estimation of these measures from a statistical point of view. The main objective here is to propose methods for estimating some of the most widely used measures of non-convexity of sets based on a random sample and to study the properties of these estimators. More specifically, after giving some geometrical background (Section~\ref{sec-back}) and formal definitions (Section \ref{mconv}),  we derive theoretical results for the estimation of four  different measures of non-convexity  (Section~\ref{sec:est}). 
One of those measures, $\mathcal L(S)$, is a novel proposal based on the probability that the convex hull of $d + 1$ randomly chosen points in $S$ is contained in $S$. We also discuss the practical implementation of the proposed estimators (Section~\ref{sec:imp}) and illustrate the applicability of the proposed estimation methods to a real data example (Section~\ref{sec:dat} and supplementary material). Finally, we provide some concluding remarks (Section~\ref{sec:con}). All proofs are given in the Appendix.

\section{Some geometric background}\label{sec-back}

This section is devoted to make explicit the notations, basic concepts and definitions 
(mostly of geometric character) that we will need in the rest of the paper.

\

\noindent\textit{Some notation}. Let us consider the $d$-dimensional Euclidean space $\mathbb{R}^d$ with the Euclidean norm $\Vert\cdot\Vert$. Given a set $S\subset \mathbb{R}^d$, we will denote by
$\mathring{S}$, $\overline{S}$, $S^c$ and $\partial S$ the interior, closure, complementary and boundary of $S$,
respectively, with respect to the usual topology of $\mathbb{R}^d$. Given a finite set $S$, we denote by $\#S$ the number of elements in $S$.
The parallel set of $S$ of radii $\varepsilon$ will be denoted as $B(S,\eps)$, that is
$B(S,\eps)=\{y\in{\mathbb R}^d:\ \inf_{x\in S}\Vert y-x\Vert\leq \eps \}$.
Given $\lambda>0$, we denote $\lambda S=\{\lambda s: s\in S\}$.
We denote $\text{diam}(S)=\sup_{x,y\in S}||x-y||$ the diameter of $S$.
If $A\subset\mathbb{R}^d$ is a Borel set, then $\mu_d(A)$ (sometimes just $\mu(A)$) will denote 
its Lebesgue measure.
We will denote by $B(x,\varepsilon)$ the closed ball
in $\mathbb{R}^d$,
of radius $\varepsilon$, centered at $x$. 
The convex hull of a set $S$ is denoted by $\text{co}(S)$. Given two points $x,y\in \mathbb{R}^d$, we denote by $\overline{xy}$ the closed segment joining them.

\

\noindent\textit{Set estimation}. Taking into account that our objective is to define estimators for certain measures of non-convexity of an unknown set $S$, based on a random sample of points supported on it, it seems natural to consider the use of plug-in type estimators. Thus, if we are able to approximate $S$ by an adequate set estimator $S_n$, we might approximate a measure of non-convexity of $S$ by the corresponding measure of non-convexity of $S_n$. We will follow this approach in some cases and, therefore, we will need some concepts from set estimation theory.

\begin{definition}\label{def-stand}  A set $S\subset \mathbb{R}^d$ is said to be standard with respect to a
	Borel measure $\nu$ in a point $x$ 	if there exists $\lambda>0$ and $\delta>0$ such that
	\begin{equation} \label{estandar}
		\nu(B(x,\eps)\cap S)\geq \delta \mu_d(B(x,\eps)),\quad 0<\eps\leq \lambda.
	\end{equation}
	A set $S\subset \mathbb{R}^d$ is said to be standard if \eqref{estandar} hold for all $x\in S$.
\end{definition}

\begin{definition}\label{rhull}
A closed set $S$ is said to be  $r$-convex if it can be expressed as the intersection of a family of complements of balls with radius $r>0$. More precisely, $S$ is $r$-convex if and only if 
\begin{equation}
	S=\bigcap_{\{y:\,B(y,r)\cap S=\emptyset\}}B(y,r)^c.\label{rconv}
\end{equation}
\end{definition}

An $r$-convex support $S$ can be estimated, from a random sample of points drawn on $S$, using the so-called $r$-convex hull of the sample. We refer to  \citep{cue12} for further results on the estimation of an $r$-convex support as well as some insights  regarding the comparison of $r$-convexity with other better known properties such as the outer $r$-sphere property. In particular, $r$-convexity is shown to be slightly stronger than the ``rolling'' outer $r$-sphere property: for every point in the boundary of $S$ there exists a ball touching that point whose interior is  included in $S^c$. Other generalizations of convexity were also studied, 
  like, for instance,  the cone-convexity \citep{chola14}, defined as follows.

\begin{definition}\label{rccc}
	A closed set $S\subset\mathbb{R}^d$ is said to be $\rho$,$h$-cone-convex by complement with parameters $\rho\in(0,\pi]$, $h>0$ if and only if
	\begin{equation}
		S=\bigcap_{\{y:\,C_{\rho,h}(y)\cap S=\emptyset\}}C_{\rho,h}(y)^c,\label{rhoccc}
	\end{equation}
where $C_{\rho,h}(y)$ denotes a finite cone with vertex at $y$, {opening angle $\rho$ and height $h$}.
\end{definition}

A $\rho$,$h$-cone-convex by complement support $S$ can be estimated, from a random sample of points drawn on $S$, using the so-called $\rho$,$h$-cone-convex hull by complement of the sample. We refer to  \citep{chola14} for further results on the estimation of a $\rho$,$h$-cone-convex by complement support and other related shape restrictions.

Some of the measures of non-convexity described in the literature make use of the Hausdorff distance or the distance in measure between sets. These distances are also used to evaluate the performance of set estimators. Let $A,C\subset \mathbb{R}^d$ be non-empty and compact. The Hausdorff distance between $A$ and $C$ is $d_H(A,C)=\max\{\sup_{a\in A}d(a,C), \ \sup_{c\in C}d(c,A)\}$. The distance in measure between  $A$ and $C$ is defined as the Lebesgue measure of their symmetric difference, that is, $d_\mu(A,C)=\mu_d(A\setminus C)+\mu_d(C\setminus A)$.

\section{Measures of non-convexity of sets} \label{mconv}
In what follows $S$ is assumed to be a non-empty, compact subset of $\mathbb{R}^d$. We give here the formal definitions of the four measures of non-convexity that we will study in detail throughout the paper. 
The definitions follow directly from the standard characterizations of convexity. Generally speaking, these measures are of geometric nature. 

\begin{enumerate}
\item First, we consider a relative version of the Hausdorff distance between the set and its convex hull.
	\begin{equation} \label{dindex}
		\mathcal{D}(S)=\frac{\sup_{s\in \text{co}(S)}d(s,S)}{\text{diam}(\text{co}(S))}.
	\end{equation}
As it happens with the Hausdorff distance between two sets, if only one point is moved far from a convex set $S$, the measure of non-convexity of the new set captures this change. 
Some properties of this index (where it is introduced without dividing by $\text{diam}(\text{co}(S)$) are studied in \citep{fradelizi2018}.
\item A relative version of the the distance in measure between the set and its convex hull is given by
	\begin{equation} \label{mindex}
		\mathcal{M}(S)= \frac{\mu_d(\text{co}(S) \setminus S)}{\mu_d(\text{co}(S))}.
	\end{equation}
	A similar index is introduced in Definition 2 in \citep{zunic2002}, where   $\mu_d(\text{co}(S) \setminus S)$ is replaced by $\mu_d(S)$, see also \citep{fradelizi2018}.
\item An equivalent condition for convexity is to ask that for all $x,y\in S$, its middle point $(x+y)/2\in S$.  Then, we can define  
	\begin{equation} \label{windex}
		\mathcal{W}(S)= 1 - \mathbb{P}((X_1+X_2)/2\in S),
	\end{equation}
being $X_1,X_2$ iid, uniformly distributed on $S$. The index $\mathcal{W}(S)$ is  a particular case of the index $\mathcal{C}_{\alpha}(S)=\mathbb{P}((\alpha X_1+(1-\alpha)X_2)\in S)$ introduced in  \citep{rahtu2006}, (observe that $\mathcal{W}(S)=1-\mathcal{C}_{1/2}(S)$).
\item  A new measure of non-convexity is defined as follows. Let $X_1,\dots,X_{d+1}$ iid, uniformly distributed on  $S$, we define
	\begin{equation} \label{LS}
	\mathcal{L}(S)=1- \E\Bigg[\frac{\mu_d(\Delta(X_1,\dots,X_{d+1})\cap S)}{\mu_d(\Delta(X_1,\dots,X_{d+1}))}\Bigg],
	\end{equation} 
	where $\Delta(X_1,\dots,X_{d+1})$ is the $d$-dimensional (open) simplex defined by $X_1,\dots,X_{d+1}$. Observe that, if $S$ is convex, then $\mathcal{L}(S)=0$.
\end{enumerate}

\

The following properties follow easily from the definitions. Those referring to $\mathcal{D}(S)$ and $\mathcal{M}(S)$ are a direct consequence of Lemmas 2.4 and 2.5 in \citep{fradelizi2018}.

\begin{proposition} \label{rem}\ Let $S$ be compact subset of $\mathbb{R}^d$.

	\begin{enumerate}
 		\item Measures $\mathcal{D}$, $\mathcal{M}$, $\mathcal{W}$ and $\mathcal{L}$ take values in $[0,1]$.
 		\item  $\mathcal{D}(S)=0$ if and only if $S$ is convex.
 		\item Under the assumption that $co(S)$ has nonempty interior,  $\mathcal{M}(S)=0$ if and only if $S$ is convex.
 		\item  In the class of closed subsets of $\mathbb{R}^d$ it holds that $\mathcal{W}(S)=0$ if and only if $S$ is convex,  except for a set of measure zero.
 		\item  In the class of closed subsets of $\mathbb{R}^d$ it holds that $\mathcal{L}(S)=0$ if and only if $S$ is convex,  except for a set of measure zero.
		\item  Measures $\mathcal{D}$, $\mathcal{M}$, $\mathcal{W}$ and $\mathcal{L}$ are translation and scale invariant.
	\end{enumerate} 
\end{proposition}

\subsection{Why consider different measures of non-convexity?}
There is no single best measure of non-convexity, and depending on which kind of convexity departures we are interested in controlling, one measure may be more appropriate than another. We briefly discuss this issue through the sets shown in Figure \ref{dibujos2}. 
The set (a) is close to be convex with respect to all measures $\mathcal{D}(S)$, $\mathcal{M}(S)$, $\mathcal{W}(S)$  and $\mathcal{L}(S)$. If we compare the sets (b) and (c), they have the same value of  $\mathcal{D}(S)$. Therefore, these two sets are indistinguishable in terms of non-convexity with respect to that measure. This is not the case if we consider, for instance, the measure $\mathcal{M}(S)$, that will be larger for the set in (c). The set shown in (d) is close to be convex with respect to $\mathcal{W}(S)$ or $\mathcal{L}(S)$ but far from being convex with respect to $\mathcal{D}(S)$ or $\mathcal{M}(S)$ .

\begin{figure}[!ht]
\begin{tikzpicture}[x=0.18cm,y=0.18cm]

\hspace{1cm}
\node[color=black] (c) at (0,7)  {$(a)$};
\filldraw[color=black, fill=lightgray](0,0) circle (4);
\filldraw[color=black, fill=white](0,0) circle (0.2);

\hspace{3cm}
\node[color=black] (c) at (0,7)  {$(b)$};
 \ngram{4}{8}{0}{black,fill=lightgray}
\draw[lightgray, fill=lightgray] (0,0) --  (67.5:4) arc(67.5:382.5:4) -- cycle;
\draw[black]  (67.5:4) arc(67.5:382.5:4) ;

\hspace{3cm}
\node[color=black] (c) at (0,7)  {$(c)$};
\ngram{4}{8}{0}{black,fill=lightgray} 

\hspace{3.5cm}
\node[color=black] (c) at (0,7)  {$(d)$};
 \filldraw[color=lightgray, fill=lightgray](0,0) circle (4);
  \draw[color=black, fill=lightgray] (3.99756, 0.139598)-- (10,0)--(3.99756, -0.139598);
  \draw[color=black, fill=lightgray] (-3.99756, 0.139598)-- (-10,0)--(-3.99756, -0.139598);
  \draw[black]  (2:4) arc(2:178:4) ;
  \draw[black]  (182:4) arc(182:358:4) ;
\end{tikzpicture}
\caption{Examples of non-convex sets}
 \label{dibujos2}
\end{figure}
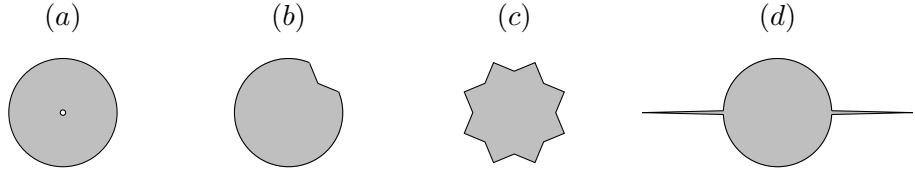

We also discuss the different behavior of the measures of non-convexity with a more detailed example. Let us consider the sets:
\begin{align*}
A_1(t)=& [-1,1]\times [-1,1] \setminus [-t,t]\times [-t,t]  \quad  \textrm{with  $t \in [0,0.9],$}\\
A_2(t)=& \left\{
 \begin{array}{ll}
 [-1,1]\times [-1,1],&{\textnormal{if}} \ \ t=0\\
 {[-1,1]\times [-1,1]} \setminus \{ (x,y):\vert y \vert \leq (x-1)/t +1  \},&{\textnormal{if}} \ \  t \in (0,2]
 \end{array}\right.\\
A_3(t)=& [-1/2,1/2]\times [-1/2,1/2] + ( \pm(t+1/2), \pm (t+1/2))	  \quad  \textrm{with  $t \in [0,1].$}
\end{align*}
Note that, for $t=0$, all sets coincide with a square of side two. In Figure \ref{comparison} we show the sets $A_i(t)$ for $t=1/2$, $i=1,2,3$. It can be proved that $\mathcal{D}(A_1(t))=t/\sqrt{8}$, $\mathcal{D}(A_2(t))=t/\sqrt{8(1+t^2)}$ and $\mathcal{D}(A_3(t))=t/(2(1+t))$. We also have $\mathcal{M}(A_1(t))=t^2$, $\mathcal{M}(A_2(t))=t/4$ and $\mathcal{M}(A_3(t))=1-1/(1+t)^2$. In Figure \ref{comparison2}, we represent $\mathcal{D}(A_i(t))$ and $\mathcal{M}(A_i(t))$ (solid and dashed line, respectively) for $i=1,2,3$. Regarding the computation of $\mathcal{W}(A_i(t))$ and $\mathcal{L}(A_i(t))$, we approximate their values by a Monte Carlo method, see Figure \ref{comparison2} (dotted and dot-dashed line, respectively).
We observe that the behavior of the measures of non-convexity varies depending on the set considered and we cannot establish, in general,  a relation of order between them. As expected, the measure $\mathcal{D}$ is more sensitive to small changes in the shape of a set than $\mathcal{M}$, as long as those changes do not substantially affect the measure of the set. Moreover, although a priori we might think that the probabilistic measures $\mathcal{W}$ and $\mathcal{L}$ have a similar behavior, this is not always the case. Note, for example, that for the set $A_3(t)$, the measure $\mathcal{W}(A_3(t))$ is constant ($\mathcal{W}(A_3(t))=3/4$) for $t\geq 1/2$, whereas $\mathcal{L}(A_3(t))$ keeps increasing as a function of $t$, thus making the sets distinguishable in terms of non-convexity.

	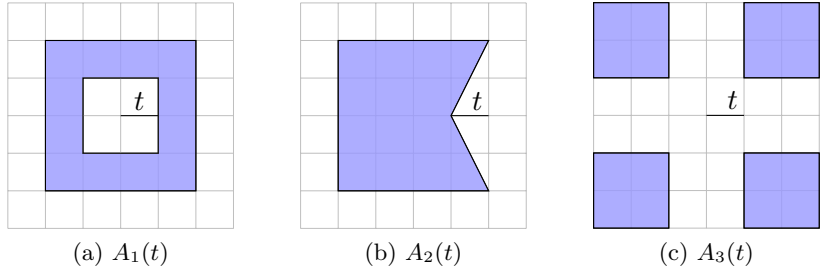
\begin{figure}[!ht]
	\vspace{0.5cm}
        \centering
        \subfloat[$A_1(t)$]{
        \begin{tikzpicture}
        \draw[step=0.5cm,gray!50,very thin] (-1.5,-1.5) grid (1.5,1.5); 
       \def\t{0.5}
       \draw[fill=blue!40, opacity=0.8] (-1,-1)--(1,-1)--(1,1)--(-1,1)-- cycle;
        \draw  (-1,-1)--(1,-1)--(1,1)--(-1,1)-- cycle;
       \draw[fill=white] (-\t,-\t)--(\t,-\t)-- (\t,\t)--(-\t,\t) -- cycle;
        \draw[step=0.5cm,gray!50,very thin] (-0.5,-0.5) grid (0.5,0.5);
         \draw(-\t,-\t)--(\t,-\t)-- (\t,\t)--(-\t,\t) -- cycle; 
        \coordinate (A) at (0,0);
\coordinate (B) at (\t,0);
\draw (A) -- (B);
\draw (\t*0.5,0.2) node {$t$};
        \end{tikzpicture}
        }\hspace{0.5cm}
        \subfloat[$A_2(t)$]{
        \begin{tikzpicture}
      \draw[step=0.5cm,gray!50,very thin] (-1.5,-1.5) grid (1.5,1.5); 
       \def\t{0.5}
       \draw[fill=blue!40, opacity=0.8] (-1,-1)--(1,-1)--(1-\t,0)--(1,1)--(-1,1)-- cycle;
       \draw (-1,-1)--(1,-1)--(1-\t,0)--(1,1)--(-1,1)-- cycle;
       \coordinate (B) at (\t,0);
             \coordinate (A) at (1,0);
\draw (A) -- (B);
\draw (\t+0.35,0.2) node {$t$};
        \end{tikzpicture}
        }\hspace{0.5cm}
        \subfloat[$A_3(t)$]{
        \begin{tikzpicture}
       \draw[step=0.5cm,gray!50,very thin] (-1.5,-1.5) grid (1.5,1.5); 
       \def\t{0.5}
       \draw[fill=blue!40, opacity=0.8] (0.5,0.5)--(1.5,0.5)--(1.5,1.5)--(0.5,1.5)-- cycle;
       \draw[fill=blue!40, opacity=0.8] (0.5,-1.5)--(1.5,-1.5)--(1.5,-0.5)--(0.5,-0.5)-- cycle;
       \draw[fill=blue!40, opacity=0.8] (-1.5,0.5)--(-0.5,0.5)--(-0.5,1.5)--(-1.5,1.5)-- cycle;
       \draw[fill=blue!40, opacity=0.8] (-1.5,-1.5)--(-0.5,-1.5)--(-0.5,-0.5)--(-1.5,-0.5)-- cycle;
       \draw (0.5,0.5)--(1.5,0.5)--(1.5,1.5)--(0.5,1.5)-- cycle;
       \draw (0.5,-1.5)--(1.5,-1.5)--(1.5,-0.5)--(0.5,-0.5)-- cycle;
       \draw(-1.5,0.5)--(-0.5,0.5)--(-0.5,1.5)--(-1.5,1.5)-- cycle;
       \draw (-1.5,-1.5)--(-0.5,-1.5)--(-0.5,-0.5)--(-1.5,-0.5)-- cycle;
              \coordinate (B) at (0,0);
             \coordinate (A) at (0.5,0);
\draw (A) -- (B);
\draw (0.35,0.2) node {$t$};
        \end{tikzpicture}
        }
        \caption{In blue, sets $A_1(t)$, $A_2(t)$ and $A_3(t)$ fot $t=1/2$.}
        \label{comparison}
    \end{figure}

  \begin{figure}[!ht]

 \centering
 \includegraphics[width=\textwidth]{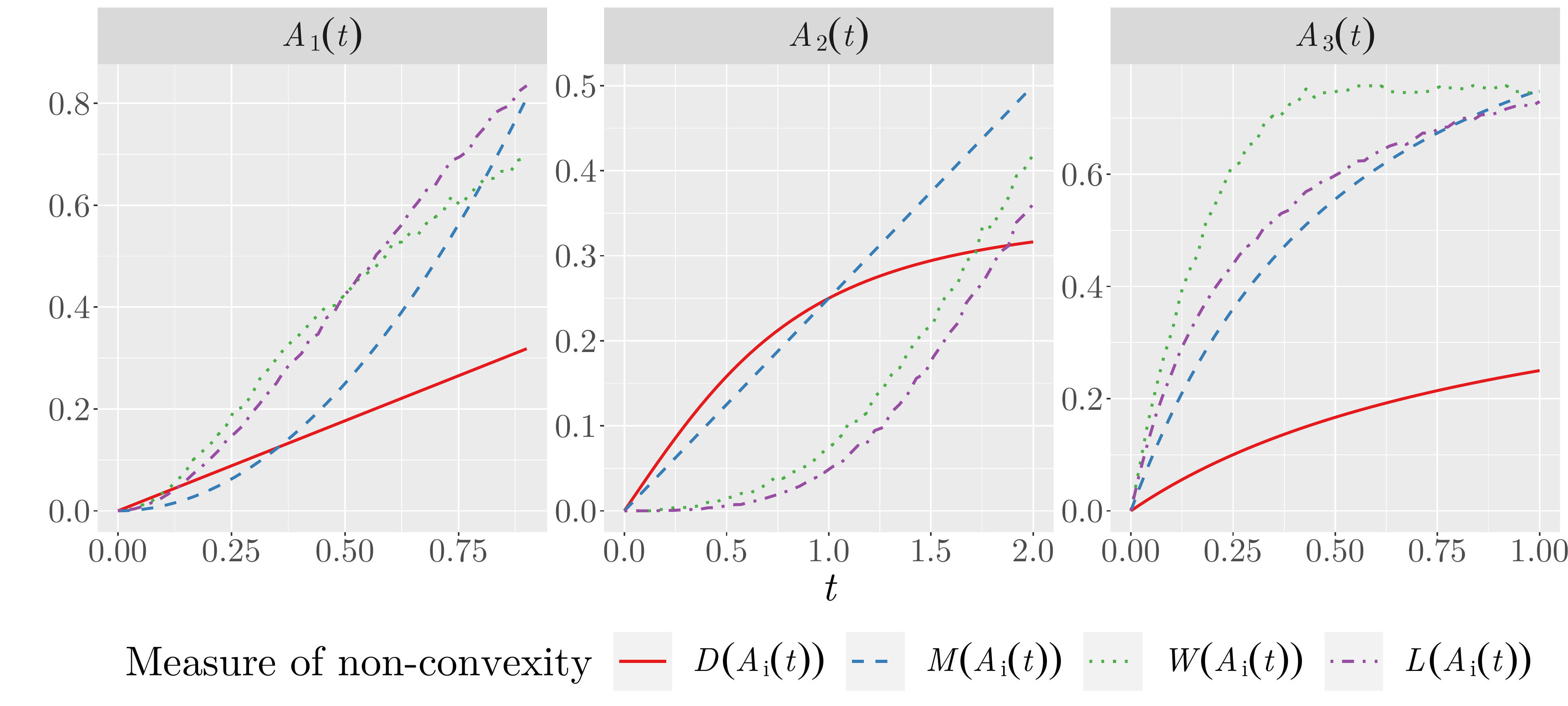}
 \caption{Measures of non-convexity for the sets $A_i(t)$, $i=1,2,3$, for different values of $t$. The measures $\mathcal{D}(A_i(t))$ and $\mathcal{M}(A_i(t))$ are calculated exactly. For each value of $t$, the measure $\mathcal{W}(A_i(t))$ is approximated by Monte Carlo, based on $5000$ pairs of uniform points drawn on $A_i(t)$.  For each value of $t$, the measure $\mathcal{L}(A_i(t))$ is approximated by Monte Carlo, based on $5000$ triples of uniform points drawn on $A_i(t)$.}\label{comparison2}
\end{figure}

 \section{Estimation of measures of non-convexity and theoretical results}\label{sec:est}
 
  Given a compact set $S \subset \mathbb R^d$, let $\aleph_n= \{X_1, \ldots, X_n\}$ be an iid sample with the same distribution $P_X$ as $X$ supported on $S$. We propose here estimators for the measures of non-convexity introduced in Section \ref{mconv}, based on the sample $\aleph_n$. We also give theoretical results on their consistency.
 
\subsection{Estimation of $\mathcal{D}(S)$}\label{estD}
 
Let us first consider the measure $\mathcal{D}(S)$, as defined in (\ref{dindex}). We propose to estimate $\mathcal{D}(S)$ by  
\begin{equation}\label{dest}
\mathcal{D}(\aleph_n)=\frac{\sup_{s\in \text{co}(\aleph_n)}d(s,\aleph_n)}{\text{diam}(\aleph_n)}.
\end{equation}
Note that we simply replace in (\ref{dindex}) the unknown set $S$ by the sample. It can be proved that $\mathcal{D}(\aleph_n)$ is a consistent estimator, as it is stated in the following theorem. 
  
\begin{theorem} \label{th1} Let $S$ be a compact set and $\aleph_n=\{X_1,\dots,X_n\}$ iid  of $X$ supported on $S$. Then, with probability one, for $n$ large enough,
	 	$$|\mathcal{D}(S)-\mathcal{D}(\aleph_n)|\leq \frac{3(1+\gamma)d_H(S,\aleph_n)}{\text{diam}(\text{co}(S))},$$
for all $\gamma>0$.
\end{theorem}

 \begin{remark}\label{remD} The asymptotic distribution of $d_H(\aleph_n,S)$  can be found in \citep{penrose2021random}, when $\partial S$ is  $C^2$. This allows to obtain an asymptotic upper bound for $|\mathcal{D}(S)-\mathcal{D}(\aleph_n)|$, and also to test (at level asymptotically smaller than a given $\alpha$)  $H_0: S\text{ is convex }$, that is $\mathcal{D}(S)=0$,  against $H_1:S$ is not convex. 
 \end{remark}

\subsection{Estimation of $\mathcal{M}(S)$}\label{estM}
 	
Let us now consider the measure $\mathcal{M}(S)$, as defined in (\ref{mindex}). Let $S_n$ be an estimator of $S$ based on $\aleph_n$. We propose to estimate $\mathcal{M}(S)$ by 
\begin{equation}\label{mest}
\hat{\mathcal{M}}(S_n)=\frac{\mu_d(\text{co}(\aleph_n)\setminus S_n)}{\mu_d(\text{co}(\aleph_n))}.
\end{equation}
\begin{theorem}\label{consistM} Let $S\subset \mathbb{R}^d$ be compact and $\aleph_n$ an iid sample on $S$. Let $S_n$ be any sequence of sets such that $d_{\mu_d}(S,S_n)\to 0$  a.s. Then, $\hat{\mathcal{M}}(S_n)\to \mathcal{M}(S)$ a.s.
\end{theorem}

\subsection{Estimation of $\mathcal{W}(S)$} \label{TG2}

Recall the definition of $\mathcal{W}(S)$ in (\ref{windex}).  Let $S_n$ be an estimator of $S$ based on $\aleph_n$. Let $(Z_{11},Z_{12}),\dots,(Z_{m1},Z_{m2})$ independent of $\aleph_n$, uniformly distributed on ${S}_n\times {S}_n$. We propose to estimate $\mathcal{W}(S)$ by
\begin{equation}\label{west}
\hat{\mathcal{W}}_m(S_n)=1-\frac{1}{m}\sum_{i=1}^m\mathbb{I}_{\{(Z_{j1}+Z_{j2})/2\in S_n\}}.
\end{equation}

We obtain the following result on the consistency of $\hat{\mathcal{W}}_m(S_n)$ when $S_n$ is the simple estimator defined by (see \citep{dev1980}),
\[{S}_n=B(\aleph_n,\eps_n)=\bigcup_{i=1}^nB(X_i,\eps_n).\]

\begin{theorem}\label{thtamura} Let $S\subset\mathbb{R}^d$ be a compact set such that $\mu_d(\partial S)=0$ and $\aleph_n=\{X_1,\dots,X_n\}$ iid uniformly distributed on $S$. 	
	Let $\eps_n\to 0$ be a sequence of positive real numbers and ${S}_n=B(\aleph_n,\eps_n)$.
Assume that $\eps_n$ is such that, with probability one, for $n$ large enough $S\subset{S}_n$. 	
Let $\Gamma_m:= \{(Z_{11},Z_{12}),\dots,(Z_{m1},Z_{m2})\}$ independent of  $\aleph_n$, uniformly distributed on ${S}_n\times {S}_n$. Then,
 ${\hat{\mathcal{W}}_m}({S}_n) \to \mathcal{W}(S)$, i.e.

$$\lim_{n \to \infty}\lim_{m \to \infty}\frac{1}{m}\sum_{j=1}^m  \mathbb{I}_{\{(Z_{j1}+Z_{j2})/2\in S_n\}}=\mathbb{P}((X_1+X_2)/2\in S),$$ 
with probability one. 	 
\end{theorem}

\subsection{Estimation of $\mathcal{L}(S)$}

Our proposal to estimate $\mathcal{L}(S)$ is based on the theory of $U$-statistics. Let ${\mathcal{L}}_n(S)$ denote the $U$-statistic of degree $d+1$ associated  to $\mathcal{L}(S)$, that is, 
\begin{equation}\label{uest} 
{\mathcal{L}}_n(S)=1-\frac{1}{\binom{n}{d+1}} \sum_{1\leq i_1<\dots <i_{d+1}\leq n} \frac{\mu_d(\Delta(X_{i_1},\dots,X_{i_{d+1}}) \cap S) }{\mu_d(\Delta(X_{i_1},\dots,X_{i_{d+1}}))},
\end{equation}
where the sum is over all possible $d$-dimensional simplexes that can be constructed taking $d+1$ points in $\aleph_n$. 
Let $S_n$ be such that $d_{\mu}(S,S_n)\to 0$ a.s. Then, we propose to estimate $\mathcal{L}(S)$ by
\begin{equation}\label{lnhat}
{\hat{\mathcal{L}}_n}({S}_n)=1-\frac{1}{\binom{n}{d+1}} \sum_{1\leq i_1<\dots <i_{d+1}\leq n} \frac{\# \mathcal{R}(X_{i_1},\dots,X_{i_{d+1}})\mu_d({S}_n)}{(n-d-1)\mu_d(\Delta(X_{i_1},\dots,X_{i_{d+1}}))},
\end{equation} 
where 
$$\mathcal{R}(X_{i_1},\dots,X_{i_{d+1}})=\big[\aleph_n\setminus \{X_{i_1},\dots,X_{i_{d+1}}\}\big]\cap \Delta(X_{i_1},\dots,X_{i_{d+1}}).$$

\begin{remark}
In case $\mu_d(S)$ is known, $\mathcal{L}(S)$ can be estimated by ${\hat{\mathcal{L}}_n}({S})$, simply replacing in (\ref{lnhat}) the measure of the estimator $\mu_d(S_n)$ by $\mu_d(S)$. Note also that, if you have a consistent estimator of $\mu_d(S)$, to obtain ${\hat{\mathcal{L}}_n}({S})$, it is not necessary to determine $S_n$.

\end{remark}

\begin{remark}
Let $X_1,\dots,X_{d+2}$ uniformly distributed on a compact set $S$. The measure $\mathcal{L}(S)$ can be rewritten as,

$$\mathcal{L}(S)=1- \mu_d(S)\E\Bigg[\frac{ \mathbb{I}_{\left\{ X_{d+2} \in \Delta(X_1,\dots,X_{d+1}) \right \}}}{\mu_d(\Delta(X_1,\dots,X_{d+1}))}\Bigg].$$
The proof is straightforward since,  
\begin{eqnarray*}
\E\Bigg[\frac{ \mathbb{I}_{\left \{ X_{d+2} \in \Delta(X_1,\dots,X_{d+1}) \right \}}}{\mu_d(\Delta(X_1,\dots,X_{d+1}))}\Bigg] &=&  
  \E \left [ \E\left( \frac{\mathbb{I}_{\left \{ X_{d+2} \in \Delta(X_1,\dots,X_{d+1}) \right \} }}{\mu_d(\Delta(X_1,\dots,X_{d+1}))}\right) \Big | X_1,\dots,X_{d+1} \right ] \\
  &=&    \E \left [  \frac{ P_X \left ( \Delta(X_1,\dots,X_{d+1}) \right ) }{\mu_d(\Delta(X_1,\dots,X_{d+1}))}  \right ] \\
  &=&    \frac{1}{\mu_d(S)} \E \left [  \frac{ \mu_d  \left ( \Delta(X_1,\dots,X_{d+1}) \cap S \right ) }{\mu_d(\Delta(X_1,\dots,X_{d+1}))}  \right ]  
\end{eqnarray*}
Hence, another possible estimator for $\mathcal{L}(S)$ could be derived from the $U$-statistic of order $(d+2)$, 
$${\mathcal{L}_n^{(d+2)}}({S})= 1 - \mu_d \left ({S} \right) \frac{1}{\binom{n}{d+2}} \sum_{1\leq i_1<\dots <i_{d+2}\leq n} \frac{\mathbb{I}_{\left \{ X_{i_{d+2}} \in \Delta \left( X_{i_1},\dots,X_{i_{d+1}} \right) \right \} }}{\mu_d(\Delta(X_{i_1},\dots,X_{i_{d+1}}))}.$$
\end{remark}

\begin{theorem}\label{consistency} Let $S\subset \mathbb{R}^d$ be a compact set and $\aleph_n=\{X_1,\dots,X_n\}$ iid  of $X$ supported on $S$.  Let $S_n$ be an estimator of $S$ such that $d_{\mu}(S,S_n)\to 0$ a.s. Then,
	 $${\hat{\mathcal{L}}_n}({S}_n) \to \mathcal{L}(S), $$
in probability.
\end{theorem}

We also have the following rates of convergence for some special classes of sets.

\begin{remark}
Let us first assume that $S\subset\mathbb{R}^d$ with $d>1$ is a compact set such that both $S$ and $S^c$ are $r$-convex (see Definition \ref{rhull}). Let $S_n$ be the so-called $r$-convex hull of $\aleph_n$. Under the hypothesis of Theorem 3 in \citep{rod07}, it can be proved that a rate of order $(\log(n)/n)^{2/(d+1)}$ can be attained in Theorem \ref{consistency}. If we only assume that $S$ is $\rho,h$-cone convex by complements (see Definition \ref{rccc}) and we use as ${S}_n$ the so-called $\rho,h$-cone convex hull of $\aleph_n$, it follows from Theorem 6 in \citep{chola14} that a rate of order $(\log(n)/n)^{1/d}$ can be attained in Theorem \ref{consistency}.
\end{remark}

\subsubsection{Asymptotic normality of $\mathcal{L}(S)$}

Using the asymptotic theory for U-statistics and a Theorem about the convergence of the empirical processes dependent on one parameter it is possible to obtain the limiting distribution for $\hat{\mathcal{L}}_n$. It can be seen in Theorem \ref{asymptdist1} how the shape of the set $S$ impacts on the dispersion of the estimator.

\begin{theorem}\label{asymptdist1} Let $S\subset \mathbb{R}^d$ be a  non convex compact set, fulfilling $\overline{\mathring{S}}=S$. Let $\aleph_n=\{X_1,\dots,X_n\}$ be an iid sample,  uniformly distributed on $S$. 	 Then
\begin{equation}\label{tcl}
 \sqrt{n} \left( {\hat{\mathcal{L}}_n}({S}_n) - \mathcal{L}(S) \right) \to N(0, d^2 \xi_1)\quad \textrm{in distribution,}
 \end{equation}
with $\xi_1= \textrm{Var} \left( h_1(X_1) \right)$, $h_1(x_1)=\E \left \{ h(x_1, X_{2}, \ldots, X_{d+1}) \right\},$ and 
$$h(X_1,\ldots,X_{d+1})= 1-\frac{\mu_d(\Delta(X_1,\dots,X_{d+1})\cap S)}{\mu_d(\Delta(X_1,\dots,X_{d+1}))}.$$ 
\end{theorem}

\section{Effective computation of non-convexity measures}\label{sec:imp}

In this section we will present some ideas and algorithms for the practical computation of some of the estimators of non-convexity measures studied in the previous sections and their implementation in R (\citep{rcran}). We restrict ourselves to the bidimensional case. Therefore, let us assume that we have $\aleph_n=\{X_1,\ldots, X_n\}$ iid sample of $X$ supported on $S\subset\mathbb{R}^2$.
We will describe how to compute $\mathcal{D}(\aleph_n)$, $\hat{\mathcal{M}}({S}_n)$, $\hat{\mathcal{W}}_m({S}_n)$ and $\hat{\mathcal{L}}_n({S}_n)$, as defined in (\ref{dest}), (\ref{mest}), (\ref{west}) and (\ref{lnhat}), respectively. 

\

\noindent{\textit{Computation of  $\mathcal{D}(\aleph_n)$}}. Recall that $\mathcal{D}(\aleph_n)=\sup_{s\in \text{co}(\aleph_n)}d(s,\aleph_n)/\text{diam}(\aleph_n)$. The exact computation of $\sup_{s\in \text{co}(\aleph_n)}d(s,\aleph_n)$ is intimately related to the so-called largest empty circle (LEC) problem, a well known problem in computational geometry, see \citep{pre1990}. It is defined on a set $P$ of $n$ points in the plane as the problem of finding the largest circle that contains no points of $P$ and whose center lies inside the convex hull of $P$. It can be solved by means of the Voronoi diagram in $O(n\log(n))$ time (the largest empty circle is proved to be always centered at either a vertex on the Voronoi diagram of $P$ or on an intersection between a Voronoi edge and the convex hull of $P$). Note that $\sup_{s\in \text{co}(\aleph_n)}d(s,\aleph_n)$ corresponds to the radius of this largest empty circle. We also refer to \citep{pre1990} for a generalization of the result to a multidimensional framework.

\

\noindent{\textit{Computation of  $\hat{\mathcal{M}}({S}_n)$}}. Recall that $\hat{\mathcal{M}}({S}_n)=\mu_d(\text{co}(\aleph_n)\setminus S_n)/\mu_d(\text{co}(\aleph_n))$, being ${S}_n$ a suitable set estimator, as described in Subsection \ref{estM}. For the discussion of the implementation we will consider as ${S}_n$ the $r$-convex hull of $\aleph_n$, denoted by $C_r(\aleph_n)$. The algorithm for computing $C_r(\aleph_n)$ can be found in \citep{ede1983}. The package \verb|alphahull| (\citep{alp25}) provides an implementation of the algorithm in R, see also \citep{pat2010}. Moreover, since $C_r(\aleph_n)\subset \text{co}(\aleph_n)$, the numerator in $\hat{\mathcal{M}}({S}_n)$ can be directly computed as the difference between the area of $\text{co}(\aleph_n)$ and that of $C_r(\aleph_n)$, whose values can be determined exactly.

\

\noindent{\textit{Computation of  $\hat{\mathcal{W}}_m({S}_n)$}}.
Recall that $\hat{\mathcal{W}}_m(S_n)=1-\frac{1}{m}\sum_{j=1}^m  \mathbb{I}_{\{((Z_{j1}+Z_{j2})/2)\in {S}_n\}}$, being ${S}_n$ a suitable set estimator of $S$, as described in Subsection \ref{TG2}, and $(Z_{11},Z_{12}),\dots,(Z_{m1},Z_{m2})$, $m$ random pairs of points, independent of $\aleph_n$, uniformly distributed on ${S}_n\times {S}_n$. In accordance with Theorem \ref{thtamura}, we consider $S_n=B(\aleph_n,\eps_n)$ for some $\eps_n>0$. Note that the points $(Z_{ji},Z_{j2}), j=1,\ldots, m,$  can be obtained using an standard acceptance-rejection method, generating a large number of points uniformly distributed on a rectangle containing $S_n$ and accepting those whose distance to $\aleph_n$ is lower than $\eps_n$. In order to make the computation faster, we use a $kd$-tree to find the point in $\aleph_n$ that is closest to each generated point. Also, the indicator function  in $\hat{\mathcal{W}}_m(S_n)$ can be evaluated  for each $j=1,\ldots,m,$  through the distance from $(Z_{j1}+Z_{j2})/2$ to $\aleph_n$. We make use of the package \verb|RANN| (\citep{rann}), providing fast nearest neighbour search, for the implementation of the estimator in R.

\

\noindent{\textit{Computation of  $\hat{\mathcal{L}}_n({S}_n)$}}. Recall the definition of $\hat{\mathcal{L}}_n({S}_n)$ in (\ref{lnhat}). For its implementation, we need an estimation of $\mu_d(S)$ (in case it is unknown) and the computation of the cardinality of 
$$\mathcal{R}(X_{i_1},\dots,X_{i_{d+1}})=\big[\aleph_n\setminus \{X_{i_1},\dots,X_{i_{d+1}}\}\big]\cap \Delta(X_{i_1},\dots,X_{i_{d+1}}).$$
Therefore, for $d=2$, we need to calculate for all triangles determined by three points in $\aleph_n$, the numbers of sample points in their interiors. Al algorithm for solving this problem can be found in \citep{epp1992}. They prove (Theorem 2.1) that a points set $\aleph_n$ in the plane can be processed in $O(n^2)$ time and space, such that afterward, for each triangle in $\aleph_n$, the number of points in it can be determined in constant time.

\section{Real Data}\label{sec:dat}

For this illustrative example we have considered images from the International Skin Imaging Collaboration (ISIC) archive, an open source platform with publicly available digital images of skin lesions\footnote{\urlstyle{same}\url{https://www.isic-archive.com/}}. ISIC datasets are widely used by researchers in the area of skin cancer detection using image analysis. Also, the annual challenges sponsored by ISIC since 2016, see \citep{gut2016}, have led to significant advances in the field. Being skin cancer a major public health problem, initiatives such as ISIC's are especially important to promote research and development of methods to improve the early detection. We refer to \citep{cas2022} for a review on scientific usage of the ISIC image datasets.
 
 In Figure \ref{isic} (top row) we show three examples of dermoscopic images ($600\times 450$ pixels) from the ISIC archive dataset\footnote{The image name of the selected examples are (a) ISIC\textunderscore 0024764, (b) ISIC\textunderscore 0024726, and (c) ISIC\textunderscore 0025586} and their corresponding binary segmentation masks (bottom panel). The images are available at \citep{tsc2018}. We refer to \citep{tsc2019} for a complete discussion on image processing techniques for automated segmentation. 
 
The three selected examples are clearly different in terms of convexity. Figure \ref{isic} (a) shows a benign skin lesion that is practically convex. It is therefore expected that any of the defined measures of non-convexity would take values close to zero in this case. On the other hand, Figure \ref{isic} (b) and (c) show skin lesions (also benign) that are clearly non-convex. In these cases it is to be expected not only that the defined measures of non-convexity take values {\textit{significantly}} larger than zero but also different from each other, thus reflecting that the way in which these two lesions depart from convexity is not the same. For instance, the lesion in Figure \ref{isic} (c) is apparently closer to be convex than the lesion in Figure \ref{isic} (b) when the indexes $\mathcal{D}(S)$ or $\mathcal{M}(S)$ are considered. However, if we consider the index $\mathcal{W}(S)$ one would expect the opposite behavior. The results in Table \ref{tab:isic} on the estimated values of these three measures of non-convexity corroborate these considerations. For each lesion (white area in each image) and each sample size ($n=1000, 2500, 5000, 10000$), we have generated 200 uniform samples (we randomly select $n$ white pixels in the binary segmentation masks and, for each selected pixel, we generate a uniformly distributed random point within that pixel). We report the mean value and standard deviation of $\mathcal{D}(\aleph_n)$, $\hat{\mathcal{M}}({S}_n)$ and $\hat{\mathcal{W}}_m({S}_n)$, over the 200 samples. For the case of $\hat{\mathcal{M}}({S}_n)$, ${S}_n$ is the $r$-convex hull of $\aleph_n$. Note that, in this case, the estimation depends on the parameter $r$, that in practice may be unknown. For  $\hat{\mathcal{W}}_m({S}_n)$, ${S}_n=B(\aleph_n,\epsilon_n)$. Again, the estimation depends on an unknown parameter $\eps_n$. In Table \ref{tab:isic} we report the results for a fixed value of $r$ ($r=7.5$) and $\eps_n$ ($\eps_n=5$). In the Supplementary material we provide, in addition, the results obtained for different values for $r$ and $\eps_n$, so that we can evaluate their impact on $\hat{\mathcal{M}}({S}_n)$ and $\hat{\mathcal{W}}_m({S}_n)$, respectively. We observe that an optimal selection for $r$ may depend on the sample size. In general, it is recommended to use small values of $r$ as $n$ increases, as the risk of {\textit{convexifying}} the set is reduced, see Figure 1 and Table 1 in  the Supplementary material.  The simulations also show a considerable stability with respect to $\eps_n$, as long as its value is not too small with respect to the sample size (as is the case for the considered values of $\eps_n$ for $n=1000$, see Figure 2 and Table 2 in  the Supplementary material).

Our goal with these simple examples is to show how to compute an approximate value for some of the discussed measures of non-convexity, using the proposed estimation methods based on random samples of points, rather than dealing with all the available pixels in the images. Also, and although this is out of the scope of our paper, these estimated measures (once computed over the whole dataset) could be included as input features for the classification of skin lesions. It would be interesting to analyze if, in combination with other relevant characteristics, some of these non-convexity measures help to increase the accuracy of the considered algorithms, see \citep{oliveira2019}.

\captionsetup{position=top}
\begin{figure}[!ht]
\begin{center}
\subfloat[]{\includegraphics[width=0.25\textwidth]{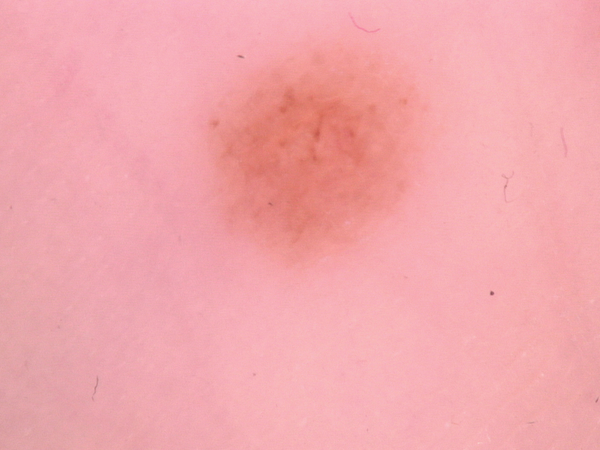}}\hspace{0.5cm}
\subfloat[]{\includegraphics[width=0.25\textwidth]{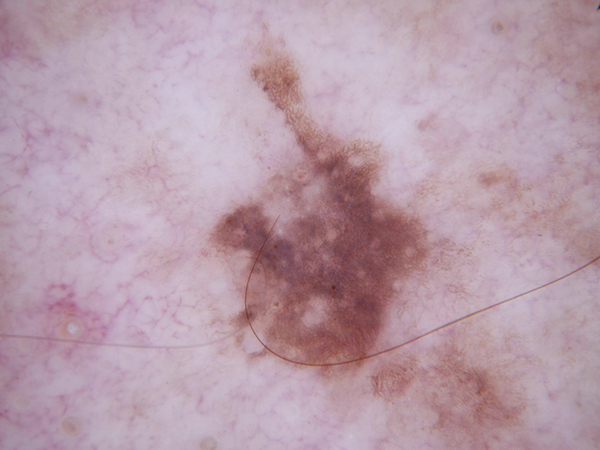}}\hspace{0.5cm}
\subfloat[]{\includegraphics[width=0.25\textwidth]{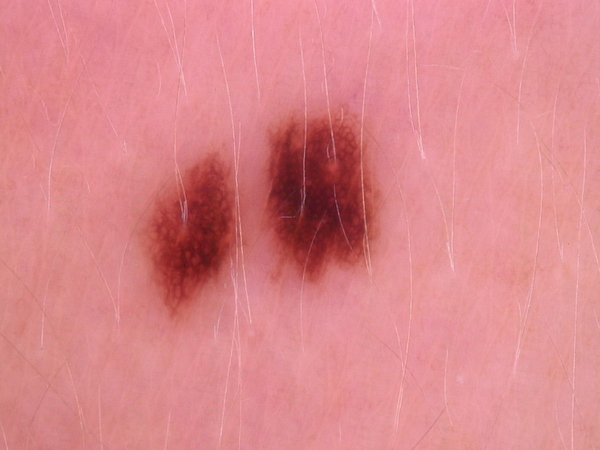}}\\ \vspace{0.2cm}
{\includegraphics[width=0.25\textwidth]{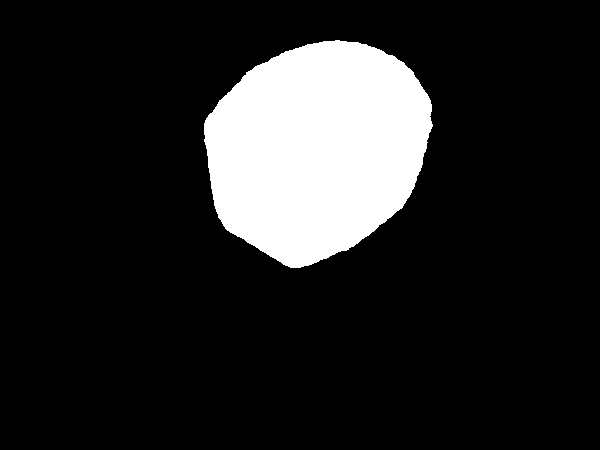}}\hspace{0.5cm}
{\includegraphics[width=0.25\textwidth]{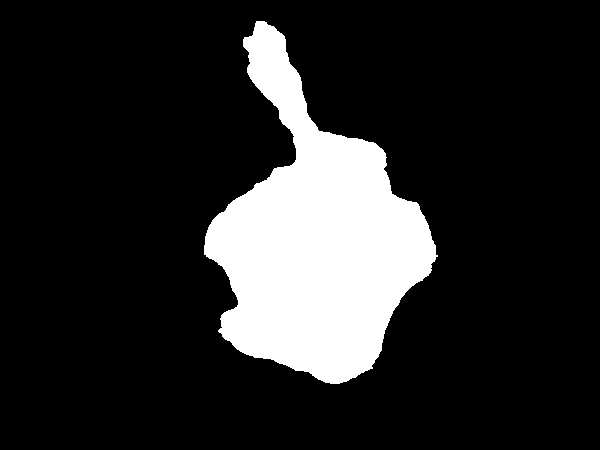}}\hspace{0.5cm}
{\includegraphics[width=0.25\textwidth]{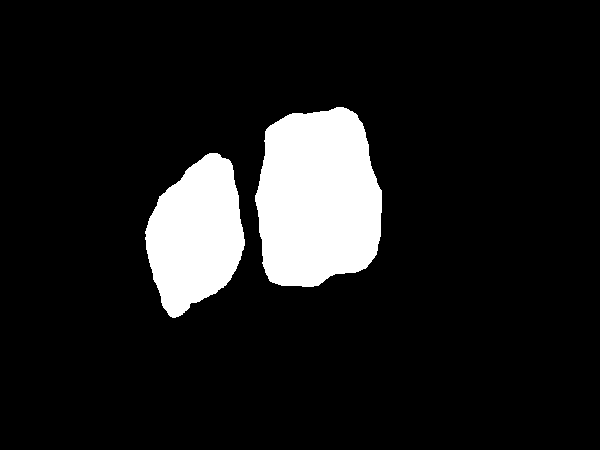}}\vspace{0.2cm}
\caption{Top row: Dermoscopic images from the ISIC archive dataset. Bottom row: Segmentantion masks}\label{isic}
\end{center}
\end{figure}

\begin{table}[!ht]
\begin{center}
\begin{tabular}{cccccccc}
&&\multicolumn{2}{c}{$\mathcal{D}(\aleph_n)$}&\multicolumn{2}{c}{$\hat{\mathcal{M}}({S}_n)$}&\multicolumn{2}{c}{$\hat{\mathcal{W}}_m({S}_n)$}\\
\cline{3-8}
Image&$n$&Mean & Sd&Mean & Sd&Mean & Sd\\
\hline
 \multirow{4}{*}{{\includegraphics[width=0.15\textwidth]{figures/ISIC_0024764_segmentation.png}}}
&1000 & 0.0467 & 0.0038 & 0.2868 & 0.0258 & 0.1277 & 0.0121 \\ 
&2500 & 0.0309 & 0.0025 & 0.0367 & 0.0041 & 0.0053 & 0.0016 \\ 
&5000 & 0.0228 & 0.0020 & 0.0231 & 0.0017 & 0.0001 & 0.0001 \\ 
&10000 & 0.0166 & 0.0014 & 0.0163 & 0.0010 & 0.0000 & 0.0000 \\ 
\hline
 \multirow{4}{*}{{\includegraphics[width=0.15\textwidth]{figures/ISIC_0024726_segmentation.png}}}
&1000 & 0.1566 & 0.0041 & 0.5042 & 0.0206 & 0.1956 & 0.0117 \\ 
&2500 & 0.1554 & 0.0025 & 0.2563 & 0.0052 & 0.0563 & 0.0024 \\ 
&5000 & 0.1545 & 0.0017 & 0.2393 & 0.0026 & 0.0471 & 0.0010 \\ 
&10000 & 0.1537 & 0.0012 & 0.2316 & 0.0017 & 0.0465 & 0.0008 \\
 \hline
  \multirow{4}{*}{{\includegraphics[width=0.15\textwidth]{figures/ISIC_0025586_segmentation.png}}}
 &1000 & 0.1087 & 0.0047 & 0.3155 & 0.0190 & 0.1918 & 0.0108 \\ 
  &2500 & 0.1057 & 0.0029 & 0.2025 & 0.0040 & 0.1121 & 0.0040 \\ 
  &5000 & 0.1048 & 0.0021 & 0.1860 & 0.0024 & 0.0992 & 0.0021 \\ 
  &10000 & 0.1040 & 0.0013 & 0.1768 & 0.0015 & 0.0920 & 0.0015 \\ 
 \hline
\end{tabular}
\caption{For each image and each sample size, mean value and standard deviation of $\mathcal{D}(\aleph_n)$, $\hat{\mathcal{M}}({S}_n)$ and $\hat{\mathcal{W}}_m({S}_n)$ over 200 uniform samples (on the white area). In $\hat{\mathcal{M}}({S}_n)$, ${S}_n=C_r(\aleph_n)$ with $r=7.5$. In $\hat{\mathcal{W}}_m({S}_n)$, ${S}_n=B(\aleph_n,\eps_n)$ with $\eps_n=5$.}\label{tab:isic}
\end{center}
\end{table}

\section{Concluding Remarks}\label{sec:con}
We have considered and compared several measures of non-convexity of a set $S\subset\mathbb{R}^d$, that attempt to quantify departures from convexity from different perspectives. Also, a new non-convexity index is introduced, which is easy to compute and provides another possible approach to the problem.  We provide consistent estimators for them based on an iid sample $\aleph_n$ on $S$. For some of them we also obtain the asymptotic distribution.
Up to our knowledge, this is the first time that this problem has been addressed from a statistical point of view. 

\

A closely related problem is that of testing whether a set is convex, based on a sample on it. This problem has been studied, for instance, in \citep{aaron2017maximal} and \citep{delicado2014data}. Although we do not focus on this problem, one could think of deriving convexity tests from some of the measures of non-convexity considered in this work, see Remark \ref{remD}.

\

A different interesting problem, which is not address in the present manuscript, is to analyze for which measures of non-convexity (when they are small), convex relaxation problems work well. Recall that in convex relaxation we replace the problem
	$$
	\mbox{minimize } f(x) \ \ \mbox{subject to} \ \ x \in S,
	$$
	by its convex relaxation, 
	$$ \mbox{minimize } f(x) \ \ \mbox{subject to} \ \ x \in S_1,
	$$
	where $S \subset S_1$ is compact, convex and $f: S_1 \to \mathbb R$ is convex and continuous in $S_1$. See for instance  \citep{zhou2021} and the references therein. Just to get an idea of this problem consider the following toy example.
For the family of non-convex sets depending on a parameter $t$ given in Figure 2, 

$$	A_1(t)= [-1,1]\times [-1,1] \setminus [-t,t]\times [-t,t]  \quad  \textrm{with  $t \in [0,0.9].$}
$$
If we take $f(x,y)=x^2 + y^2$ and $S=\text{co}(A_1(t))$ the minimum in $S$ is zero and is attained at the point $(0,0)$, while the minimum at $A_1(t)$ is attained at the points
		$(t,0),(0,t),(-t,0), (0, -t)$ and takes the value $t^2$.


\section*{Appendix A}

In this section we provide the proofs of the theoretical results presented in the manuscript.\\

\begin{proof}[Proof of Theorem \ref{th1}]

Since $\text{co}(\aleph_n)\subset \text{co}(S)$,
\begin{align*}
	\sup_{x\in  \text{co}(\aleph_n)} d(x,\aleph_n)\leq & \ \sup_{s\in \text{co}(S)} d(s,S)+d_H(\aleph_n,S)\\
	= &\  \text{diam}(\text{co}(S))\mathcal{D}(S)-\text{diam}(\text{co}(\aleph_n))\mathcal{D}(S) +\\
	&\text{diam}(\text{co}(\aleph_n))\mathcal{D}(S)+d_H(\aleph_n,S).
\end{align*}
Since $0<\text{diam}(\text{co}(S))-\text{diam}(\text{co}(\aleph_n))\leq 2d_H(\aleph_n,S)$ and $\mathcal{D}(S)\leq 1$,
\begin{align*}
\mathcal{D}(\aleph_n)-\mathcal{D}(S)\leq& \ \frac{2\mathcal{D}(S)d_H(\aleph_n,S)}{\text{diam}(\text{co}(\aleph_n))}+\frac{d_H(\aleph_n,S)}{\text{diam}(\text{co}(\aleph_n))}\\
\leq & \
\frac{2d_H(\aleph_n,S)}{\text{diam}(\text{co}(S))-2d_H(\aleph_n,S)}+\frac{d_H(\aleph_n,S)}{\text{diam}(\text{co}(S))-2d_H(\aleph_n,S)}.
\end{align*}
From $d_H(\aleph_n,S)\to 0$ it follows that for all $\gamma>0$,
$$\mathcal{D}(\aleph_n)-\mathcal{D}(S)\leq \frac{3(1+\gamma)d_H(\aleph_n,S)}{\text{diam}(\text{co}(S))}.$$ 
To bound $\mathcal{D}(S)-\mathcal{D}(\aleph_n)$ let   $s\in \text{co}(S)$ and $s_n\in \text{co}(\aleph_n)$ the closest point to $s$.
 
	\begin{align*}
		d(s,S)\leq & \ d(s,s_n)+d(s_n,S)\\
		\leq &\ d(s,s_n)+d(s_n,\aleph_n)+d_H(\aleph_n,S)\\
		\leq &\  d(s,s_n)+\sup_{s'\in \text{co}(\aleph_n)}d(s',\aleph_n)+d_H(\aleph_n,S)\\
		= &\ d(s,s_n)+\text{diam}(\text{co}(\aleph_n))\mathcal{D}(\aleph_n)+d_H(\aleph_n,S).
	\end{align*}
From $d_H(\text{co}(S),\text{co}(\aleph_n))\leq d_H(S,\aleph_n)$ it follows that $d(s,s_n)\leq d_H(\aleph_n,S)$. 
	Then
\begin{align*}
\sup_{s\in \text{co}(S)} d(s,S)\leq &\ 2d_H(S,\aleph_n)+\text{diam}(\text{co}(\aleph_n))\mathcal{D}(\aleph_n)\\
\leq &\ 2d_H(S,\aleph_n)+\text{diam}(\text{co}(S))\mathcal{D}(\aleph_n).
\end{align*} 
Finally,
	$$\mathcal{D}(S)-\mathcal{D}(\aleph_n)\leq \frac{2d_H(S,\aleph_n)}{\text{diam}(\text{co}(S))}.$$	
\end{proof}


\begin{proof}[Proof of Theorem  \ref{consistM}]
From $d_\mu(\text{co}(\aleph_n),\text{co}(S))=\mathcal{O}(d_H(\aleph_n,S))$  (see the proof of Theorem 3 in \citep{aaron2017maximal}), it follows that $\mu(\text{co}(\aleph_n))\to \mu(\text{co}(S))$. Then, from Slutsky's Theorem, we have to prove that 	
\begin{equation}\label{thMeq1}
\mu(\text{co}(\aleph_n) \setminus S_n)\to \mu(\text{co}(S)\setminus S).
\end{equation}
Let us split it into two terms,
$$\mu(\text{co}(\aleph_n)\cap S_n^c)=\mu(\text{co}(\aleph_n)\cap S_n^c\cap S)+\mu(\text{co}(\aleph_n)\cap S_n^c\cap S^c)=:I+II.$$
Since $I\leq d_\mu(S,S_n)\to 0$, we have to prove that $II\to \mu(\text{co}(S)\setminus S)$. 
$$II=\mu(\text{co}(\aleph_n)\cap S^c)-\mu(\text{co}(\aleph_n)\cap S^c\cap S_n)=:A+B.$$
Since $B\leq d_\mu(S,S_n)\to 0$, \eqref{thMeq1} follows from $d_\mu(\text{co}(\aleph_n),\text{co}(S))\to 0$.
\end{proof}

\begin{proof}[Proof of Theorem \ref{thtamura}]	

 Let us assume, without loss of generality, that $\mu_d(S)=1$.
 	Let $Z_{n1},Z_{n2}$ uniformly distributed in $S_n=B(\aleph_n, \epsilon_n)$. 
	By the Law of Large Numbers, conditionally to $S_n$ we have that for $n$ fixed 
	$$\lim_{m \to \infty} \frac{1}{m}\sum_{j=1}^m  \mathbb{I}_{\{(Z_{j1}+Z_{j2})/2\in S_n \vert S_n\}} =  \mathbb{P}((Z_{n1}+Z_{n2})/2\in S_n \vert S_n)\quad a.s..$$
Now,
\begin{multline}
	\mathbb{P}((Z_{n1}+Z_{n2})/2\in S_n|S_n)=\mathbb{P}((Z_{n1}+Z_{n2})/2\in S_n\setminus S|S_n)+\\
	\mathbb{P}((Z_{n1}+Z_{n2})/2\in S|S_n,Z_{n1},Z_{n2}\in S)\mathbb{P}(Z_{n1},Z_{n2}\in S)
\end{multline}
	Let $X_1,X_2$ uniformly distributed on $S$, since  $S\subset S_n$ eventually almost surely,  we have that
	$$\mathbb{P}((Z_{n1}+Z_{n2})/2\in S|S_n,Z_{n1},Z_{n2}\in S)=\mathbb{P}((X_1+X_2)/2\in S).$$
Also, 
	$$\mathbb{P}((Z_1+Z_2)/2\in S_n\setminus S|S_n)\leq \mathbb{P}((Z_1+Z_2)/2\in B(S,\epsilon_n) \setminus S|S_n,\{Z_{n1},Z_{n2}\in S\})$$ 
On the other hand,
	$$\mathbb{P}((Z_{n1}+Z_{n2})/2\in B(S,\epsilon_n) \setminus S|S_n, \{Z_{n1},Z_{n2}\in S\})=\mathbb{P}((X_1+X_2)/2\in B(S,\epsilon_n) \setminus S)$$
where $X_1,X_2$ is uniformly distributed on $S$. 	
	$$\mathbb{P}((X_1+X_2)/2\in B(S,\epsilon_n) \setminus S)\to \mathbb{P}((X_1+X_2)/2\in \partial S)=0.$$
\end{proof}


\begin{proof}[Proof of Theorem  \ref{consistency}]

Since $|{\hat{\mathcal{L}}_n}({S}_n)-{\hat{\mathcal{L}}_n}(S)|\leq d_\mu(S,{S}_n)$, in order to prove that ${\hat{\mathcal{L}}_n}({S}_n)\to \mathcal{L}(S)$ in probability it is enough to prove that ${\hat{\mathcal{L}}_n}(S) \to \mathcal{L}(S)\text{ in probability}$. In fact we will prove that $\sqrt{n}({\hat{\mathcal{L}}_n}(S)-{\mathcal{L}}_n(S))\to 0$ in probability,  which will be used in the proof of Theorem \ref{asymptdist1}. Recall that  ${\mathcal{L}}_n(S)$ denotes the $U$-statistic associated  to $\mathcal{L}(S)$, see (\ref{uest}), with kernel 
$$h(X_1,\ldots,X_{d+1})= 1-\frac{\mu_d(\Delta(X_1,\dots,X_{d+1})\cap S)}{\mu_d(\Delta(X_1,\dots,X_{d+1}))}.$$ 
Let us write
$$\hat{\mathcal{L}}_n(S)= {\mathcal{L}}_n(S) + \Big(\hat{\mathcal{L}}_n(S)-{\mathcal{L}}_n(S) \Big).$$
The kernel function $h$ is bounded. Then, it follows from Section 5.4 (Theorem A) in \citep{serfling2009} that,  
		$${\mathcal{L}}_n(S) \ccs  \mathcal{L}(S).$$		
To simplify the notation, we write $\Delta$ instead of $\Delta(X_{i_1},\dots,X_{i_{d+1}})$. Then,
		\begin{align*}
			\hat{\mathcal{L}}_n(S)-{\mathcal{L}}_n(S)= & \frac{1}{\binom{n}{d+1}} \sum_{1\leq i_1<\dots <i_{d+1}\leq n} \frac{\frac{\mu_d(S)}{n-d-1} \# \mathcal{R}(X_{i_1},\dots,X_{i_{d+1}}) -\mu_d(\Delta\cap S)}{\mu_d(\Delta)}.
		\end{align*}
		For all $\delta>0$,
		\begin{small}
			\begin{multline*}
				\mathbb{P}(|\hat{\mathcal{L}}_n(S)-{\mathcal{L}}_n(S)|>\epsilon)\leq \\ \mathbb{P}\Bigg(\sup_{\Delta}\Bigg|\frac{\mu_d(\Delta \cap S)-\frac{\mu_d(S)}{n-d-1} \# \mathcal{R}(X_{i_1},\dots,X_{i_{d+1}})}{\mu_d(\Delta)}\Bigg|>\epsilon\Bigg|\frac{\mu_d(\Delta)}{\mu_d(S)}>\delta\Bigg)\\
				+ 	\mathbb{P}(\mu_d(\Delta)/\mu_d(S)<\delta).
			\end{multline*}
		\end{small}
		Let $\mathcal{A}_\delta$ be the class of all $d$-dimensional simplexes of $\mathbb{R}^d$ whose vertices belongs to $S$ and $\mu_d(\Delta)/\mu_d(S)>\delta$. From Corollary 13.1 in \citep{luc}  $\mathcal{A}_0$ has finite Vapnik-Chervonenkis dimension, denoted by $V_{\mathcal{A}_0}$ and $V_{\mathcal{A}_\delta}<V_{\mathcal{A}_0}$.
		\begin{multline*}
			\mathbb{P}(|\hat{\mathcal{L}}_n(S)-{\mathcal{L}}_n(S)|>\epsilon)\leq \\ \mathbb{P}\Bigg(\sup_{\Delta \in \mathcal{A}_\delta}\Bigg|\frac{\mu_d(\Delta \cap S)-\frac{\mu_d(S)}{n-d-1}\# \mathcal{R}(X_{i_1},\dots,X_{i_{d+1}}) }{\mu_d(\Delta)}\Bigg|>\epsilon\Bigg)\\
			+ 	\mathbb{P}(\mu_d(\Delta)/\mu_d(S)<\delta).
		\end{multline*}		
		Let $\nu$ denote the uniform measure on $S$, that is, $\nu(B)=\mu_d(B\cap S)/\mu_d(S)$. Its empirical version based on an iid sample, $X_1,\dots,X_{n-d-1}$, is
		denoted by $\nu_{n-d-1}(B)$.
		Then for all $\delta>0$.
		\begin{align*}
			\mathbb{P}(|\hat{\mathcal{L}}_n(S)-{\mathcal{L}}_n(S)|>\epsilon)\leq &
		\ 	\mathbb{P}\Bigg(\sup_{\Delta\in  \mathcal{A}_\delta}\Big|\nu(\Delta)-\nu_{n-d-1}(\Delta)\Big|>\epsilon \mu_d(\Delta)/\mu_d(S)\Bigg)\\
			&\ \hspace{2cm} + 			\mathbb{P}(\mu_d(\Delta)/\mu_d(S)<\delta)\\
						\leq &\  			(n-d)^{V_{\mathcal{A}_0}}\exp(-C_1\delta^2 \epsilon^2(n-d-1))\\
				&\ \hspace{2cm} 			+\mathbb{P}(\mu_d(\Delta)/\mu_d(S)<\delta),
		\end{align*}
for some $C_1>0$. Observe that $\mathbb{P}(\mu_d(\Delta)/\mu_d(S)<\delta)\to 0$ as $\delta\to 0$. To get $\sqrt{n}|\hat{\mathcal{L}}_n(S)-{\mathcal{L}}_n(S)|\to 0$ in probability it is enough to take $\delta_n\to 0$ such that $\sqrt{n}\delta_n^2/\log(n)\to \infty$.			
\end{proof}

\begin{proof}[Proof of Theorem \ref{asymptdist1}]

First we will prove \eqref{tcl}  for $\hat{\mathcal{L}}_n(S)$, that is, assuming that $\mu_d(S)$ is known. Since $\sqrt{n}|\hat{\mathcal{L}}_n(S)-{\mathcal{L}}_n(S)|\to 0$ in probability, from Slutsky's Theorem it is enough to prove $\sqrt{n} \left( {{\mathcal{L}}_n}({S}) - \mathcal{L}(S) \right) \to N(0, d^2 \xi_1)$ in distribution. 
		 We have that  $0 \leq h \leq 1$, which entails that  $\mathbb{E}(h^2)< \infty $  and $h_1$ is non negative.
	Since $S$ is not convex there exists $x_1,x_2\in S$ such that $\overline{x_1x_2}$ the segment joining $x_1$ and $x_2$ is not included in $S$. 
	Let $t\in \overline{x_1x_2}\cap S^c$ such that $d(t,S)>\delta>0$ for some $\delta>0$. Since $S=\overline{\mathring{S}}$ there exists $v_1,v_2\in \mathring{S}$ and $||x_i-v_i||<\delta/4$, $i=1,2$. Let $x,x_4,\dots,x_{d+1} \in \mathring{S}$ and $0<\delta_1<\delta/4$ such that $B(x_i,\delta_1)\subset S$  for all $i=4,\dots,d+1$, $B(v_1,\delta_1)\subset \mathring{S}$, $B(v_2,\delta_1)\subset \mathring{S}$  and $B(x,\delta_1)\subset S$.
	Let 
	$$a_1\in B(x,\delta_1), a_2\in B(v_1,\delta_1), a_3\in B(v_2,\delta_1),a_4\in B(x_4,\delta_1),\dots,a_{d+1}\in B(x_{d+1},\delta_1)$$
	and $\Delta=\Delta(a_1,\dots,a_{d+1})$. Then there exists $\ell$, not depending on $\Delta$, such that
	\begin{equation} \label{cota1}
		\frac{\mu_d(B(t,\delta)\cap \Delta)}{\mu_d(\Delta)}>\ell >0.
	\end{equation}
	Let 
	$$A=\{\omega:X_1(\omega)\in B(v_1,\delta_1), X_2(\omega)\in B(v_2,\delta_1),X_i(\omega)\in B(x_i,\delta_1), i=4,\dots,d+1\}.$$
	If $\omega \in A$ then, from \eqref{cota1} and $S\cap \Delta\subset B(t,\delta)^c\cap \Delta$, for all $x\in \mathring{S}$,
\begin{multline*}
	\frac{\mu_d(\Delta(x,X_1,X_2,X_4,\dots,X_{d+1})\cap S)}{\mu_d(\Delta(x,X_1,X_2,X_4,\dots,X_{d+1}))}\leq \frac{\mu_d(\Delta(x,X_1,X_2,X_4,\dots,X_{d+1})\cap B(t,\delta)^c)}{\mu_d(\Delta(x,X_1,X_2,X_4\dots,X_{d+1}))}\\
	<1-\ell<1.
\end{multline*}
	Then, for all $x\in \mathring{S}$,
	\begin{align*}
		\mathbb{P}(h(x,X_1, X_{2}, \ldots, X_{d})<1)>&\  \mathbb{P}(A)\\
		\geq&\ (\mu_d(B(0,\delta_1))/\mu_d(S))^{d}>0.
	\end{align*} 
	This leads to $\xi_1= \textrm{Var} \left( h_1(X) \right)>0$ and from Section 5.5 (Theorem A)  of \citep{serfling2009} it follows that 	
	$$ \sqrt{n} \left( {\mathcal{L}}_n(S)- \mathcal{L}(S) \right) \to N(0, d^2 \xi_1) \quad \textrm{in distribution.}$$
	Since $\sqrt{n}|\hat{\mathcal{L}}_n(S)-{\mathcal{L}}_n(S)|\to 0$ in probability, from Slutsky's Theorem we proved \eqref{tcl}  for $\hat{\mathcal{L}}_n(S)$. 
		Now let us prove \eqref{tcl} in the general case, that is, when $\mu_d(S)$ is unknown. We will use the same technique used to prove Proposition 2.1 in \citep{boente1995asymptotic}.
	Let  $[\tau_1,\tau_2]$ be an interval containing 1, let us define 
	$$U_n(\tau)=\sqrt{n}\Big(\tau\tilde{\mathcal{L}}(S)-\tau \mathcal{L}(S)\Big).$$ 
	Now to prove the general case it is enough to prove that there exists a continuous process $U(\tau)$ such that 
	\begin{equation}\label{w}
		U_n(\tau)\stackrel{w}{\to} U(\tau),
	\end{equation} 
	where $\stackrel{w}{\to}$ means weak convergence  on the space $C([\tau_1,\tau_2])$ of continuous functions defined on $[\tau_1,\tau_2]$ endowed with the uniform norm. Observe that we have proved that $U_n(\tau)\to  N(0,(d\tau)^2\xi_1)$ in distribution, as $n\to \infty$, for all $\tau\in [\tau_1,\tau_2]$. Then from Theorem 8.1 in \citep{bill68} to prove \eqref{w} it is enough to prove that the process $U_n$ is tight. To prove that, from Theorem 12.2 in \citep{bill68} it is enough to prove 1) $U_n(1)$ is tight, and 2) there exists $\beta$ such that, for all $n$, and for all $\tau,\tau'\in [\tau_1,\tau_2]$,  $\mathbb{E}(U_n(\tau)-U_n(\tau'))^2\leq \beta(\tau-\tau')^2$. Point 1 follows from $U_n(1)\to N(0,d^2\xi_1)$. To prove point 2, observe that
	$$\mathbb{E}(U_n(\tau)-U_n(\tau'))^2=(\tau-\tau')^2 \mathbb{E}(U_n(1)^2).$$
The proof will be concluded by showing that $\sup_n \mathbb{E}(U_n(1)^2)<\infty$. Let us consider the $U$-statistic $1-\tilde{\mathcal{L}}(S)$, with kernel $1-h(X_1,\dots,X_{d+1})$. Then, $\mathbb{V}(1-\tilde{\mathcal{L}}(S))=(1/n)\mathbb{E}(U_n(1)^2)$. From Theorem 5.2 in \citep{hoeffding1992class}, $\lim_n  n\mathbb{V}((1-\tilde{\mathcal{L}}(S)))<\infty$, then 
$\sup_n \mathbb{E}(U_n(1)^2)<\infty$.

\end{proof}

  \begin{center}
	{\LARGE\bf SUPPLEMENTARY MATERIAL\\ \vspace{.2cm} Statistical analysis of measures of non-convexity}
\end{center}

Here we present some additional figures and tables for Section 6.

\begin{figure}[!ht]
\begin{center}
	\begin{tikzpicture}
		\node (image) at (0,0) {
			{\includegraphics[width=0.65\textwidth]{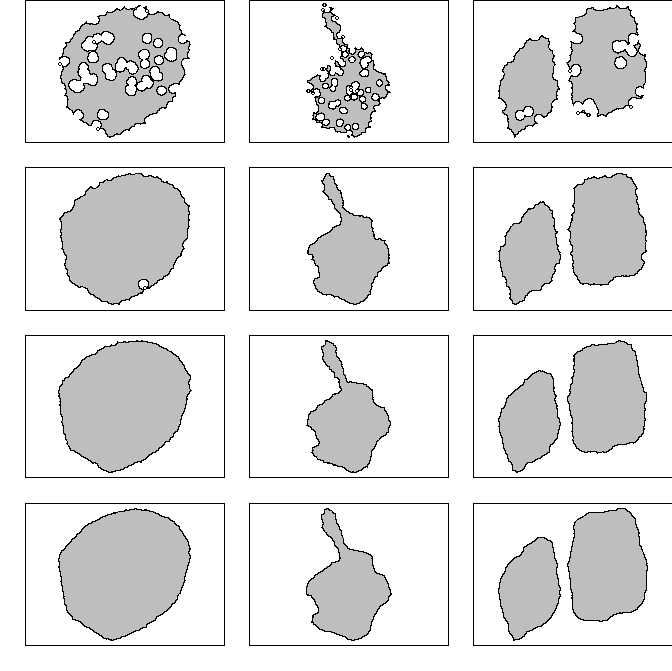}}
		};
		\node at (-3.5,5.8) {(a)};
		\node at (0.2,5.8) {(b)};
		\node at (3.7,5.8) {(c)};
		
		
		\node[rotate=90] at (-5.5,4.2) {{\small{$n=1000$}}};
		\node[rotate=90] at (-5.5,1.4) {{\small{$n=2500$}}};
		\node[rotate=90] at (-5.5,-1.3) {{\small{$n=5000$}}};
		\node[rotate=90] at (-5.5,-4) {{\small{$n=10000$}}};
	\end{tikzpicture}
	\caption{In gray, ${S}_n=C_r(\aleph_n)$, based on a sample of $n$ uniform points (from top row to bottom row, $n=1000,2500,5000,10000$) within the white area of the corresponding binary segmentation masks. In all cases, $r=7.5$.}\label{Mrhull8}
\end{center}
\end{figure}

\begin{figure}[!ht]
\begin{center}
	\begin{tikzpicture}
		\node (image) at (0,0) {
			{\includegraphics[width=0.65\textwidth]{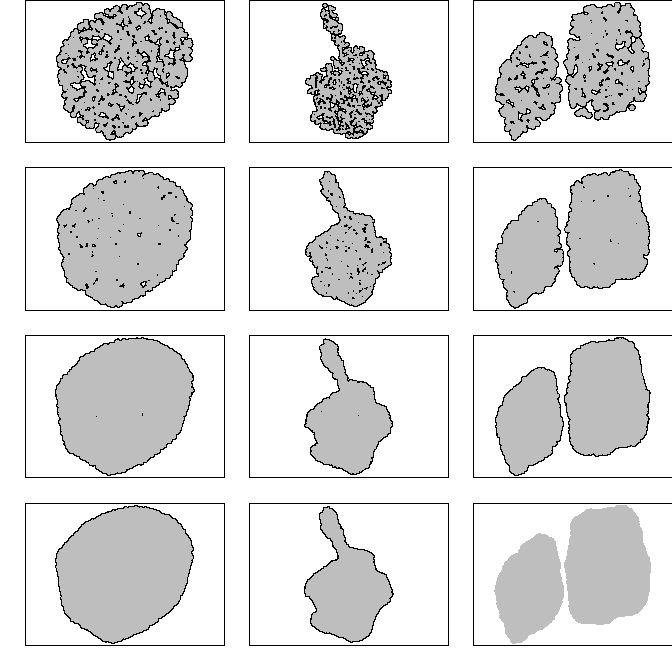}}
		};
		\node at (-3.5,5.8) {(a)};
		\node at (0.2,5.8) {(b)};
		\node at (3.7,5.8) {(c)};
		
		\node[rotate=90] at (-5.5,4.2) {{\small{$n=1000$}}};
		\node[rotate=90] at (-5.5,1.4) {{\small{$n=2500$}}};
		\node[rotate=90] at (-5.5,-1.3) {{\small{$n=5000$}}};
		\node[rotate=90] at (-5.5,-4) {{\small{$n=10000$}}};

	\end{tikzpicture}
	\caption{In gray, ${S}_n=B(\aleph_n,\epsilon_n)$, based on a sample of $n$ uniform points (from top row to bottom row, $n=1000,2500,5000,10000$) within the white area of the corresponding binary segmentation masks. In all cases, $\eps_n=5$}\label{Wdw}
\end{center}
\end{figure}

\begin{table}[!ht]
\begin{center}
	\rotatebox{90}{\begin{tabular}{cccccccc}
			&&\multicolumn{6}{c}{$\hat{\mathcal{M}}(S_n)$}\\
			\cline{3-8}
			&&\multicolumn{2}{c}{$r=10$}&\multicolumn{2}{c}{$r=8$}&\multicolumn{2}{c}{$r=5$}\\
			\cline{3-8}
			Image&$n$&Mean & Sd&Mean & Sd&Mean & Sd\\
			\hline
			\multirow{4}{*}{{\includegraphics[width=0.15\textwidth]{figures/ISIC_0024764_segmentation.png}}}
			& 1000 & 0.8115 & 0.0110 & 0.2868 & 0.0258 & 0.0731 & 0.0128 \\ 
			& 2500 & 0.2008 & 0.0171 & 0.0367 & 0.0041 & 0.0299 & 0.0027 \\ 
			& 5000 & 0.0314 & 0.0031 & 0.0231 & 0.0017 & 0.0200 & 0.0015 \\ 
			& 10000 & 0.0194 & 0.0012 & 0.0163 & 0.0010 & 0.0145 & 0.0010 \\
			
			\hline
			\multirow{4}{*}{{\includegraphics[width=0.15\textwidth]{figures/ISIC_0024726_segmentation.png}}}
			& 1000 & 0.8848 & 0.0077 & 0.5042 & 0.0206 & 0.3007 & 0.0149 \\ 
			& 2500 & 0.4291 & 0.0147 & 0.2563 & 0.0052 & 0.2471 & 0.0042 \\ 
			& 5000 & 0.2517 & 0.0041 & 0.2393 & 0.0026 & 0.2352 & 0.0026 \\ 
			& 10000 & 0.2355 & 0.0017 & 0.2316 & 0.0017 & 0.2290 & 0.0017 \\
			\hline
			\multirow{4}{*}{{\includegraphics[width=0.15\textwidth]{figures/ISIC_0025586_segmentation.png}}}
			& 1000 & 0.7315 & 0.0147 & 0.3155 & 0.0190 & 0.2332 & 0.0082 \\ 
			& 2500 & 0.2564 & 0.0118 & 0.2025 & 0.0040 & 0.1934 & 0.0052 \\ 
			& 5000 & 0.1929 & 0.0027 & 0.1860 & 0.0024 & 0.1716 & 0.0062 \\ 
			& 10000 & 0.1803 & 0.0016 & 0.1768 & 0.0015 & 0.1529 & 0.0046 \\ 
			
			\hline
	\end{tabular}}
	\caption{For each image and each sample size, mean value and standard deviation of $\hat{\mathcal{M}}({S}_n)$ over 200 uniform samples (on the white area), being $S_n=C_r(\aleph_n)$.}\label{tab:isic2b}
\end{center}
\end{table}

\begin{table}[!ht]
\begin{center}
	\rotatebox{90}{\begin{tabular}{cccccccccc}
			&&\multicolumn{8}{c}{$\hat{\mathcal{W}}_m(S_n)$}\\
			\cline{3-10}
			&&\multicolumn{2}{c}{$\eps_n=5.5$}&\multicolumn{2}{c}{$\eps_n=5$}&\multicolumn{2}{c}{$\eps_n=4.5$}&\multicolumn{2}{c}{$\eps_n=4$}\\
			\cline{3-10}
			Image&$n$&Mean & Sd&Mean & Sd&Mean & Sd&Mean & Sd\\
			\hline
			\multirow{4}{*}{{\includegraphics[width=0.15\textwidth]{figures/ISIC_0024764_segmentation.png}}}
			&1000 & 0.0838 & 0.0099 & 0.1277 & 0.0121 & 0.1901 & 0.0124 & 0.2690 & 0.0128 \\ 
			&2500 & 0.0018 & 0.0009 & 0.0053 & 0.0016 & 0.0142 & 0.0029 & 0.0348 & 0.0045 \\ 
			&5000 & 0.0000 & 0.0000 & 0.0001 & 0.0001 & 0.0002 & 0.0002 & 0.0010 & 0.0005 \\ 
			&10000 & 0.0000 & 0.0000 & 0.0000 & 0.0000 & 0.0000 & 0.0000 & 0.0000 & 0.0000 \\ 
			
			\hline
			\multirow{4}{*}{{\includegraphics[width=0.15\textwidth]{figures/ISIC_0024726_segmentation.png}}}
			&1000 & 0.1472 & 0.0105 & 0.1956 & 0.0117 & 0.2573 & 0.0121 & 0.3356 & 0.0119 \\ 
			&2500 & 0.0504 & 0.0019 & 0.0563 & 0.0024 & 0.0690 & 0.0037 & 0.0949 & 0.0053 \\ 
			&5000 & 0.0467 & 0.0010 & 0.0471 & 0.0010 & 0.0480 & 0.0011 & 0.0501 & 0.0012 \\ 
			&10000 & 0.0460 & 0.0008 & 0.0465 & 0.0008 & 0.0470 & 0.0008 & 0.0475 & 0.0008 \\
			\hline
			\multirow{4}{*}{{\includegraphics[width=0.15\textwidth]{figures/ISIC_0025586_segmentation.png}}}
			&1000 & 0.1589 & 0.0093 & 0.1918 & 0.0108 & 0.2372 & 0.0110 & 0.2989 & 0.0117 \\ 
			&2500 & 0.1027 & 0.0040 & 0.1121 & 0.0040 & 0.1236 & 0.0043 & 0.1398 & 0.0048 \\ 
			&5000 & 0.0911 & 0.0023 & 0.0992 & 0.0021 & 0.1079 & 0.0024 & 0.1165 & 0.0024 \\ 
			&10000 & 0.0840 & 0.0015 & 0.0920 & 0.0015 & 0.1001 & 0.0015 & 0.1082 & 0.0016 \\ 
			\hline
	\end{tabular}}
	\caption{For each image and each sample size, mean value and standard deviation of $\hat{\mathcal{W}}_m({S}_n)$ over 200 uniform samples (on the white area), being $S_n=B(\aleph_n,\eps_n)$.}\label{tab:isic2}
\end{center}
\end{table}


\end{document}